\documentclass[12pt,reqno]{amsart}
\usepackage{amsmath, amsfonts, epsfig, amssymb, graphics}
\usepackage{graphicx}
\usepackage[usenames]{color}
\usepackage{hyperref}
\numberwithin{equation}{section}
\newtheorem{theo}{Theorem}
\newtheorem{prop}[theo]{Proposition}
\newtheorem{conj}{Conjecture}

\newtheorem{coro}[theo]{Corollary}
\newtheorem{rema}{Remark}

 \font\German=eufm10 scaled 1100
 \font\german=eufm10 scaled 750

\newenvironment{Remark}{\begin{rema}\rm}{\end{rema}}

\def\BH{\mbf{H}}

\def\BM{\mbf{M}}

\def\mbf#1{{\mathbf{#1}}}
\def\C{{\mathbb C}}

\def\CF{{\mathbb F}}
\def\CD{{\mathcal D}}
\def\N{{\mathbb N}}

\def\S{{\mathbb S}}
\def\S{\hbox{\German S}}
\def\s{\hbox{\german S}}

\def\R{{\mathbb R}}

\def\R{{\mathbb R}}
\def\BK{\mathbb{K}}
\def\FT{\mathrm{FT}}
\def\LT{\mathrm{LT}}

\title[Steenrod operators]{Harmonics for deformed Steenrod operators}
\author{Fran\c{c}ois Bergeron, Adriano Garsia and Nolan Wallach}
\address[F. Bergeron]{D\'epartement de Math\'ematiques\\ Universit\'e
  du Qu\'ebec \`a Montr\'eal\\ Montr\'eal, Qu\'ebec, H3C 3P8, CANADA}
\email{bergeron.francois@uqam.ca}
\date{\today}
\thanks{F. Bergeron is supported by NSERC-Canada and FQRNT-Qu\'ebec.}
\begin{document}
\maketitle
\begin{abstract}
  We explore in this paper the spaces of common zeros of several deformations of Steenrod operators. 
\end{abstract}

{ \parskip=0pt\footnotesize \tableofcontents}
\parskip=8pt

\section{Introduction}
In recent years many authors have studied variations on a striking classical result of invariant theory holding for any finite group $W$ of real $n\times n$ matrices generated by reflections. Roughly stated, this result asserts that there is a natural decomposition 
  \begin{equation}\label{decomp_W}
       \R[\mbf{x}]\simeq \R[\mbf{x}]^W\otimes \R[\mbf{x}]_W
  \end{equation}
of the ring of polynomials $\R[\mbf{x}]$, in $n$ variables $\mbf{x}=x_1,x_2,\ldots,x_n$, as a tensor product of the ring $\R[\mbf{x}]^W$ of $W$-invariant polynomials, and  the ``$W$-coinvariant-space'' $\R[\mbf{x}]_W$. This last  is simply the space obtained as the quotient
of the ring $\R[\mbf{x}]$ by the ideal generated by constant-term-free 
$W$-invariant polynomials. One of the many interesting reformulations of (\ref{decomp_W}) goes through the notion of $W$-{\em harmonic polynomials}, i.e.: the polynomials $f(\mbf{x})$ that satisfy all partial differential equations of the form
 \begin{equation}\label{diff_eq}
     p(\partial_{1},\partial_{2}, \ldots, \partial_{n}) f(\mbf{x})=0,
 \end{equation}
with $p(\partial_\mbf{x})=p(\partial_{1},\partial_{2}, \ldots, \partial_{n})$ varying in the set of differential operators obtained by replacing each variable $x_i$ by the derivation with respect to this variable, here denoted $\partial_{i}$,  in constant-term-free  $W$-invariant polynomial. Indeed, one can show that the solution set of (\ref{diff_eq}) is isomorphic as a $W$-module to $\R[\mbf{x}]_W$. Further striking facts about the space $\mathcal{H}_W$ of $W$-harmonic polynomials are that
 \begin{enumerate}
    \item[(a)] $\mathcal{H}_W$ has dimension equal to the order of $W$.
    \item[(b)] $\mathcal{H}_W$ is the span of all partial derivative (of all orders) of the single polynomial 
       \begin{equation}\label{Delta_W}
            \Delta_W(\mbf{x}):=
          \det\left(\frac{\partial f_j(\mbf{x})}{\partial x_i}\right)_{1\leq i,j\leq n}
       \end{equation}
where the $f_j$'s vary in any $n$-set of algebraically independent generators of the ring of $W$-invariant polynomials.
\end{enumerate}
In the special case of $W$ equal to the symmetric group $\S_n$, the conditions in (\ref{diff_eq}) are equivalent to $f(\mbf{x})$ being a solution of the system 
\begin{equation}\label{systeme_Sn}
    \begin{matrix}
      (\partial_{1}+\ldots+\partial_{n}) f(\mbf{x}) &=& 0\\[4pt]
      (\partial^2_{1}+\ldots+\partial^2_{n}) f(\mbf{x}) &=& 0\\
         &\vdots&\\
      (\partial^n_{1}+\ldots+\partial^n_{n}) f(\mbf{x}) &=& 0
  \end{matrix}
\end{equation}
Observe that the second equation corresponds to the vanishing of the laplacian of $f(\mbf{x})$. Hence solutions of (\ref{systeme_Sn}) are indeed harmonic functions in the usual sense. However, the extra conditions appearing in (\ref{systeme_Sn}) clearly make for a stronger notion. It follows that, in the symmetric group case, the polynomial $\Delta_{\S_n}$ is simply the Vandermonde determinant
   $$\Delta_{\S_n}(\mbf{x})=\prod_{i<j} (x_i-x_j).$$
The purpose of this work is to study twisted versions of system (\ref{systeme_Sn}). More precisely, instead of $\partial^k_{1}+\ldots+\partial^k_{n}$, we will rather consider operators of the form
\begin{equation}\label{operateur}
  D_k:=\sum_{i=1}^n a_{i,k}x_i \partial_{i}^{k+1} + b_{i,k} \partial_{i}^k,
\end{equation}
 with some parameters $a_{i,k}$ and $b_{i,k}$. We then consider the solution set $\mathcal{H}_{\mbf{x}}$ of the system of partial differential equations
\begin{equation}\label{harm_ab}
      D_k f(\mbf{x}) =0,\qquad k\geq 1.
  \end{equation}
The index $\mbf{x}$, in $\mathcal{H}_{\mbf{x}}$, is to underline that we are dealing with polynomials in the $n$ variables $\mbf{x}=x_1,\ldots, x_n$. Observe that  the operators  $D_k$  are homogeneous. We say that they are of degree $-k$ since they lower degree of polynopmials by $k$. It follows that $\mathcal{H}_{\mbf{x}}$ is  graded by degree, and that it has the direct sum decomposition
   $$ \mathcal{H}_{\mbf{x}}  = \bigoplus_{d\ge 0} \pi_d(\mathcal{H}_{\mbf{x}} ), $$
with $\pi_d$ denoting the projection onto the homogeneous component of degree $d$.  In particular, it makes sense to talk about the {\em Hilbert series} 
\begin{equation}\label{hilbert_series}
    H_n(t):=\sum_{d\geq 0} t^d \dim(\pi_d(\mathcal{H}_{\mbf{x}} )),
 \end{equation}
of the space ${\mathcal{H}}_{\mbf{x}}$. Clearly the right-hand side of (\ref{hilbert_series}) depends only on $n$, the number of variables, rather than on the actual set variables. It also implicitly depends on the choice of the parameters $a_{i,k}$ and $b_{i,k}$.
Recall that the Hilbert series of the space ${\mathcal{H}}_{\s_n}$, of $\S_n$-harmonic polynomials (which corresponds to setting $a_{i,k}=0$ and $b_{i,k}=1$)  is the classical $t$-analog of $n!$:
\begin{equation}\label{t_fact}
    [n]!_t:=\prod_{k=1}^n \frac{1-t^k}{1-t}= (1+t)(1+t+t^2)\cdots (1+t+\ldots+t^{n-1}).
  \end{equation}
As we will see later, this is a ``generic'' value for $H_n(t)$.

\section{Dual operators and hit-polynomials}
Before going on with our discussion, let us consider an interesting dual point of view.  Following a terminology of Wood~\cite{wood}, we shall say that a polynomial is a {\em hit-polynomial} if it can be expressed in the form
\begin{equation}\label{hit_polynomial}
    f(\mbf{x}) = \sum_{k} D^*_k\, g_k(\mbf{x}),
 \end{equation}
for some polynomials $g_k(\mbf{x})$, with $D^*_k$ standing for the  dual operator of $D_k$ with respect to the following scalar product on  the ring of polynomials.

For two polynomials $f$ and $g$ in $\R[\mbf{x}]$, one sets
\begin{equation}\label{scalar_def}
  \langle f,g\rangle := f(\partial_{\mbf{x}})g(\mbf{x})\big|_{\mbf{x}=0}.
\end{equation}
In other words, this corresponds to the constant term of the polynomial resulting from the application of the differential operator $f(\partial_{\mbf{x}})$ to $g(\mbf{x})$. A straightforward compution reveals that, for two monomials $\mbf{x}^\mbf{a}$ and $\mbf{x}^\mbf{b}$, we have
    $$\langle \mbf{x}^\mbf{a},\mbf{x}^\mbf{b}\rangle =
      \begin{cases}
      \mbf{a}!& \text{if},\ \mbf{x}^\mbf{a}=\mbf{x}^\mbf{b},\\
      0 & \text{otherwise},
\end{cases}$$
where, as is now almost usual, $\mbf{a}!$ stands for $a_1!a_2!\cdots a_n!$. This observation makes it clear that (\ref{scalar_def}) indeed defines a scalar product on $\R[\mbf{x}]$. Moreover, the dual of the operator $\partial_i^k$ is easily checked to be multiplication by $x_i^k$. It follows that
   $$
D^*_k= \sum_{i=1}^n a_{i,k}\,x_i^{k+1}\partial_{i}+b_{i,k}\,x_i^k.
$$
From general basic linear algebra principles, it follows that the space of ${\mathcal{H}}_{\mbf{x}}$, of general harmonic polynomials, is orthogonal to the space of hit-polynomials. Moreover, since the subspace of hit-polynomials is homogeneous, the corresponding quotient $\mbf{C}$ of $\R[\mbf{x}]$, by this subspace, is isomorphic to ${\mathcal{H}}_{\mbf{x}}$ as a graded space.

For a composition $\alpha=(\alpha_1,\alpha_2,\,\ldots\,,\alpha_m)$, let us set
$$
D^*_\alpha =D^*_{\alpha_1} D^*_{\alpha_2} \cdots D^*_{\alpha_m} 
$$
Note that the operator $D^*_\alpha $, as it acts on polynomials, increases degree by 
   $$|\alpha|:=a_1+a_2+\cdots+a_m.$$
 We write $\alpha\models n$, whenever $\alpha$ is a composition of $n$, i.e.: $|\alpha|=n$. We shall refer to $|\alpha|$ as the  {\em degree } of $D^*_\alpha$. Clearly partitions are special instances of compositions. Recall that  write $\lambda\vdash n$ to say that $\lambda$ is a partition of $n$. We denote $\lambda(\alpha)$ the partition obtained by  rearranging the parts of a composition $\alpha$ in decreasing size order. This given  we have

\begin{prop}\label{red_comp}
Let the operators $D^*_k$  be such that
\begin{equation}\label{condition}
     [D_k^*,D_j^*]=c_{k,j}\,D_{k+j}^*,
 \end{equation}
for some constants $c_{k,j}$.
Then, for all $d\ge 1$ the family of operators 
$$
\{D^*_\lambda\}_{\lambda\vdash d},
$$
 spans the vector space spanned by the collection
$
\{D^*_\alpha\}_{\alpha\models  d}.
$ 
\end{prop}

\begin{proof}[\bf Proof.] 
We  show recursively that any composition $\alpha\models d$ can be expanded as\begin{equation}\label{res_red_comp}
D^*_\alpha=\sum_{ \mu} \gamma_{\alpha,\mu}\,D^*_\mu 
\end{equation}
assuming that such an  expansion exists when $\alpha$ has less than $m$ parts or less than $r$ inversions. Evidently, when $\alpha$ has $0$ inversions, or just one part, it is a partition and we have nothing to show. When $\alpha$ has $r$ inversions and $m$ parts, suppose that one of its inversions occurs at
the $i^{th }$ part of $\alpha$. This is to say that we have $\alpha_i<\alpha_{i+1}$, and we can use (\ref{condition}) in the form
$$
D^*_{\alpha_i}D^*_{\alpha_{i+1}}=D^*_{\alpha_{i+1}}D^*_{\alpha_i}+c_{\alpha_i,\alpha_{i+1}}\,D^*_{\alpha_i+\alpha_{i+1}}
$$
to rewrite $D^*_\alpha= A\, D^*_{\alpha_{i}}D^*_{\alpha_{i+1}}\, B$ as
\begin{equation}\label{reecriture}
   D^*_\alpha = A\,D^*_{\alpha_{i+1}}D^*_{\alpha_i}\, B+
     c_{\alpha_i,\alpha_{i+1}}\,A\, D^*_{\alpha_i+\alpha_{i+1}}\,B
\end{equation}
with $A:=D^*_{\alpha_1}\cdots D^*_{\alpha_{i-1}}$ and $B:=D^*_{\alpha_{i+2}}\cdots D^*_{\alpha_{m}}$.
We need only observe that the first term of the right-hand side of (\ref{reecriture}) corresponds to a composition with  less than $r$ inversions, while  the second term
has less than $m$ factors. We thus need only apply the induction hypothesis to each term to complete the proof.
\end{proof}

\section{$q$-Steenrod operators and $q$-harmonics}\label{sectionq}
For our first exploration of the spaces ${\mathcal{H}}_{\mbf{x}}$, we consider the special case corresponding to setting
    $$a_{i,k}=q,\qquad \mathrm{and}\qquad b_{i,k}=1,$$
for all $i$ and $k$. The resulting space is henceforth denoted  ${\mathcal{H}}_{\mbf{x};q}$. An element of ${\mathcal{H}}_{\mbf{x};q}$ is said to be a $q$-harmonic (polynomial). Observe that for $q=0$ we get back the more classical notion of $\S_n$-harmonic polynomials. In other words ${\mathcal{H}}_{\mbf{x};0}={\mathcal{H}}_{\s_n}$. The space  ${\mathcal{H}}_{\mbf{x};q}$ has been introduced by  Hivert and Thi\'ery (in~\cite{hivert_thiery}) as a generalisation of spaces considered by Wood in~\cite{wood}. Using the notation
 $$D_{k;q}:=\sum_{i=1}^nq\,x_i \partial_{i}^{k+1} +  \partial_{i}^k,$$
it is interesting to observe (as did Hivert and Thi\'ery) that  
 \begin{equation}\label{crochet_dq}
    [D_{k;q}, D_{j;q}]=  q(k-j) D_{k+j;q},
\end{equation}
where $[A,B]$ stands for the usual Lie bracket $AB-BA$ of operators. Indeed, this operator identity readily implies that, whenever $q\not=0$,  ${\mathcal{H}}_{\mbf{x};q}$ is simply characterized as the solution set of just the two equations:
  $$D_{1;q} f(\mbf{x})=0,\qquad {\rm and}\qquad D_{2;q}f(\mbf{x}) =0,
  $$
since all other equations, with $k>2$, follow from successive applications of (\ref{crochet_dq}), whose left-hand side involves operators with lower values of $k$.

Recall that the ring of polynomials $\R[\mbf{x}]$ can be considered as a $\S_n$-module for the action that corresponds to permutation of the variables. This action restricts to a natural $\S_n$-action on  the space ${\mathcal{H}}_{\mbf{x};q}$, since the operators $D_{k;q}$ are symmetric. It is classical that ${\mathcal{H}}_{\s_n}={\mathcal{H}}_{\mbf{x};0}$ is isomorphic, as a $\S_n$-module, to the regular representation of $\S_n$. A conjecture of  Hivert and Thi\'ery, states that this is also the case for ${\mathcal{H}}_{\mbf{x};q}$. 

\begin{conj}[Hivert-Thi\'ery]\label{wood}
   As $\S_n$-modules, the spaces ${\mathcal{H}}_{\mbf{x};q}$ is isomorphic to ${\mathcal{H}}_{\s_n}$, when  $q>0$. In particular, this implies that the Hilbert series of ${\mathcal{H}}_{\mbf{x};q}$ is $[n]!_t$.
 \end{conj}
Formula (\ref{t_fact}) is but a shadow of a finer formula describing the decomposition of each homogeneous component of ${\mathcal{H}}_{\mbf{x};q}$ into irreducible representations. As is well known, such decompositions can be expressed in a nice compact format for all graded invariant $\S_n$-modules 
   $$\mathcal{V}=\bigoplus_{d\geq 0} \mathcal{V}_d,$$
through the  {\em graded Frobenius characteristic}:
   $$\mathcal{F}_\mathcal{V}(t):=\sum_{d\geq 0} 
           \frac{1}{n!}\sum_{\sigma\in\S_n} \chi^{\mathcal{V}_d}(\sigma)
              p_{\lambda(\sigma)}.$$
 Here $\chi^{\mathcal{V}_d}$ stands for the character of the homogeneous invariant subspace $\mathcal{V}_d$, and $\lambda(\sigma)$ is the partition of $n$ corresponding to the cycle decomposition of the permutation $\sigma$. Moreover, as in Macdonald~\cite{macdonald}, $p_\lambda$ denotes the {\em power sum} symmetric function. Recall that one recovers the required graded irreducible decomposition of $\mathcal{V}$ through the expansion of $\mathcal{F}_\mathcal{V}(t)$ in the Schur function basis: 
    $$\mathcal{F}_\mathcal{V}(t)=
       \sum_{\lambda\vdash n} n_\lambda(t)\,s_\lambda.$$
The coefficient of $s_\lambda$ in $\mathcal{F}_\mathcal{V}(t)$ is the  series (or polynomial, when $\mathcal{V}$ is of finite dimension)
    $$n_\lambda(t) =\sum_{d\geq 0} n_{\lambda,d}\, t^d,$$
 such that $n_{\lambda,d}$ is the multiplicity of the irreducible representation indexed by $\lambda$ in the homogeneous component 
$\mathcal{V}_d$.

It follows from (\ref{decomp_W}) and conjecture~\ref{wood} that the graded Frobenius characteristic $F_n(t)$ of ${\mathcal{H}}_{\mbf{x};q}$ (and ${\mathcal{H}}_{\s_n}$) is
  \begin{equation}\label{regular}
    F_n(t)=[n]!_t\,(1-t)^n  \sum_{\lambda\vdash n} 
           \prod_{k=1}^{n} 
               \frac{1}{d_k!}\left(\frac{p_k}{k\,(1-t^k)}\right)^{d_k},
  \end{equation}
where $d_k=d_k(\lambda)$ is the number of size $k$ parts of $\lambda$. Thus, for $n=3$, we get
\begin{eqnarray*}
    [3]!_t(1-t)^3\left( \frac{1}{3!} \left( \frac{p_1}{1-t}\right)^3+
                \frac{p_1}{1-t}\,\frac{p_2}{2(1-t^2)} +
                \frac{p_3}{3(1-t^3)} \right)&=& \\
       && \hskip-240pt
        \frac{p_1^3}{6}(1+t)(1+t+t^2) + \frac{p_1p_2}{2}(1-t)(1+t+t^2)
        +\frac{p_3}{3}(1-t)^2(1+t)
  \end{eqnarray*}
which expands as
  $$s_3+t(1+t)s_{21}+t^3\,s_{111},$$
in term of Schur functions. In other words, the space ${\mathcal{H}}_{\s_3}$ contains one copy of the trivial representation (encoded by $s_3$)  in its $0$ degree homogeneous component, one of the sign representation (encoded by $s_{111}$)  in its degree $3$ component, and one $s_{21}$ encoded representation in both its degree $1$ and $2$ components.
          
The graded Frobenius characteristic $F_n(t)$  of the harmonics of $\S_n$  can also be written (see \cite{macdonald}) in the form
\begin{equation}\label{avec_maj}
  F_n(t)= \sum_{\lambda\vdash n}s_\lambda
       \sum_{\tau\in ST(\lambda)}t^{{\mathrm co}(\tau)}
\end{equation}
where the inner sum is over all standard Young tableaux of shape $\lambda$, and ${\mathrm co}(\tau)$ stands for the {\em cocharge} of a tableau $\tau$. If $n_\lambda$ stands for the number of standard Young tableaux of shape $\lambda$, then formula (\ref{avec_maj}) and a classical symmetric function identity, gives
\begin{eqnarray*}  
   F_n(1)&=& \sum_{\lambda\vdash n}n_\lambda\,s_\lambda\\
             &=& s_1^n.
\end{eqnarray*}
Recall that there is a natural indexing of irreducible representations of $\S_n$, by partitions $\lambda$ of $n$, such that the dimension of the irreducible representation indexed $\lambda$ is $n_\lambda$.

\section{Tilde-Harmonics and Hat-Harmonics}\label{sectiontilde}
We now consider another interesting special case that is somewhat dual to that of the last section.
Namely, we  suppose that all $b_{i,k}$'s vanish, and all $a_{i,k}$'s are equal to $1$. In formula, we consider the space of common zeros of the operators 
  $$\widetilde{D}_k:=\mbf{x}\cdot {\partial_{\mbf{x}}}^{k+1}=
     \sum_{i=1}^n x_i\,\partial_i^{k+1},$$
which is called the space of {\em tilde-harmonics}, and denoted $\widetilde{\mathcal{H}}_{\mbf{x}}$. Notice here our use of dot-product notation between the ``vectors'' $\mbf{x}$ and $\partial_\mbf{x}=(\partial_1,\ldots,\partial_n)$. It will reappear below. Just as for (\ref{crochet_dq}), we easily check that
 \begin{equation}\label{crochet_dtilde}
   [ \widetilde{D}_k, \widetilde{D}_j]= 
       (k-j) \widetilde{D}_{k+j},
\end{equation}
hence $\widetilde{\mathcal{H}}_{\mbf{x}}$ is also simply the set of common zeros of the two equations
  $$ \widetilde{D}_1  f(\mbf{x})=0,\qquad {\rm and}
     \qquad  \widetilde{D}_2  f(\mbf{x}) =0.$$
Once again, $\widetilde{\mathcal{H}}_{\mbf{x}}$ affords a natural action of the symmetric group since the operators $ \widetilde{D}_k$ are symmetric. The associated graded Frobenius characteristic is denoted $\widetilde{F}_n(t)$. 

Computer experimentations reveal a number of surprising facts about this space.  Let us denote $\widetilde{\mathcal{H}}_n(t)$ the associated Hilbert series, so that
\begin{eqnarray*}
\widetilde{\mathcal H}_2(t)&=& 1+2t+t^2+t^3,\\
\widetilde{\mathcal H}_3(t)&=& 1+3t+3t^2+4t^3+2t^4+2t^5+t^6,\\
\widetilde{\mathcal H}_4(t)&=&1+4t+6t^2+10t^3+9t^4+11t^5+9t^6+6t^7+5t^8+3t^9+t^{10}.
\end{eqnarray*} 
Specializing $t=1$, we get the global dimension of $\widetilde{\mathcal{H}}_{\mbf{x}}$ for which it is not to hard to guess, after more explicit computations, that we apparently have
  \begin{equation}\label{eq_conj_un}
     \dim \widetilde{\mathcal{H}}_{\mbf{x}}= \sum_{k=0}^n \frac{n!}{ k!}.
  \end{equation}
Of course,  the Hilbert series of  $\widetilde{\mathcal{H}}_{\mbf{x}}$ has to have some $t$-analogue of this. More experiments, and a result present below, suggest that the right candidate should be
\begin{equation}\label{conj_hilb}
   \widetilde{\mathcal H}_n(t)= \sum_{k=0}^n{n \choose k}t^k[k]_t!.
 \end{equation}
Indeed, modulo one further natural conjecture, this follows from a very explicit description of  $\widetilde{\mathcal{H}}_{\mbf{x}}$ outlined below. To state it we need one more family of operators
and yet another version of harmonic polynomials. For each $k \ge 1$ consider the operator
  \begin{equation}\label{def_dhat}
    \widehat{D}_k :=\partial_\mbf{x}^{k+1}\cdot \mbf{x}=\sum_{i=1}^n \partial_{i}^{k+1}x_i
  \end{equation}	
 which are alternatively described as
 $$\widehat{D}_k= \sum_{i=1}^n x_i \partial_{i}^{k+1} + (k+1)\, \partial_i^k.$$
Now, we introduce the space
   $$ \widehat{\mathcal{H}}_{\mbf{x}} := \big\{f(\mbf{x})\in\R[\mbf{x}]\ |\   \widehat{D}_kf(\mbf{x})= 0,\quad  \forall k\ge 1 \big\},$$ 
whose elements are said to be  ``{\em hat-harmonics}''. We will soon relate the two notions of tilde and hat harmonics. Experimentation suggest that $\widehat{\mathcal{H}}_{\mbf{x}}$  has dimension  $n!$, and that even more precisely we have the following.

\begin{conj}\label{conj_deux}
 As a graded $\S_n$-module, $\widehat {\mathcal{H}}_{\mbf{x}}$ is isomorphic to the space of $\S_n$-harmonics.
 \end{conj}

Now, for any given $k$-subset $\mbf{y}$ of the $n$ variables $\mbf{x}$,  let us consider the space $\widehat{\BH}_{\mbf{y}}$, and write
     $$e_\mbf{y}:=\prod_{x\in \mbf{y}} x,$$
for the elementary symmetric polynomial of degree $k$ in the variables $\mbf{y}$. As usual, we define the {\em support} of a monomial to be the set of variable that appear in it, with non-zero exponent. Clearly, $\mbf{y}^\mbf{a}$ has support $\mbf{y}$ if and only if $\mbf{y}^\mbf{a}=e_\mbf{y}\, \mbf{y}^\mbf{b}$, for some $\mbf{b}$.
We can now state the following remarkable fact.

\begin{theo}\label{thm_quatre}
 The space of tilde-harmonics has the direct sum decomposition
 \begin{equation}\label{decomp_star}
         \widetilde{\mathcal{H}}_{\mbf{x}}=\bigoplus_{\mbf{y}\subseteq \mbf{x} } e_\mbf{y}\widehat {\mathcal{H}}_{\mbf{y}},
  \end{equation}
if we consider that hat-harmonics for $\mbf{y}=\emptyset$ are simply the scalars.
\end{theo}

\begin{proof}[\bf Proof.]  Let us first observe that  any polynomial $f$ in $\R[\mbf{x}]$ decomposes uniquely in the form
    \begin{equation}\label{decomp_support}
          f =  \sum_{\mbf{y}\subseteq \mbf{x}}e_\mbf{y}f_\mbf{y}
  \end{equation}
with $f_{\mbf{y}}$ in $\R[\mbf{y}]$. This decomposition is simply obtained by collecting terms with respect to support. If all $f_{\mbf{y}}$ are chosen to respectively lie in $\widehat {\mathcal{H}}_{\mbf{y}}$, then $f$ is in $\widetilde{\mathcal{H}}_{\mbf{x}}$. Indeed,
we can easily check the operator identity
\begin{equation}\label{eq_prop_un}
         \widetilde{D}_k \,e_\mbf{x} = e_\mbf{x}\,\widehat{D}_k,
  \end{equation}
 where $e_{\mbf{x}}$ stands for the operator of multiplication by $e_{\mbf{x}}$.
Observe that, for all $k$,  we get
\begin{eqnarray*} 
(\mbf{x}\cdot {\partial_{\mbf{x}}}^{k+1})\,f&=&
 \sum_{\mbf{y}\subseteq \mbf{x}} (\mbf{y}\cdot {\partial_{\mbf{y}}}^{k+1})\,e_\mbf{y}f_\mbf{y} \\ 
&=& \sum_{\mbf{y}\subseteq \mbf{x}} e_\mbf{y}\,({\partial_{\mbf{y}}}^{k+1}\cdot \mbf{y})\,f_\mbf{y}.
\end{eqnarray*}
Hence $f$ is in $\widetilde{\mathcal{H}}_{\mbf{x}}$, if and only if each $f_\mbf{y}$ lies in $\widehat {\mathcal{H}}_{\mbf{y}}$. In other words, we get
\begin{equation}\label{inclusion}
 \bigoplus_{\mbf{y}\subseteq \mbf{x}} e_\mbf{y}\widehat {\mathcal{H}}_{\mbf{y}}\ =\  \widetilde{\mathcal{H}}_{\mbf{x}},
\end{equation}
thus finishing the proof.
\end{proof}
Note that this argument goes through, exactly in the same maner, if we consider the more general case of operators  $a_k\,\widetilde{D}_k$ and $a_k\,\widehat{D}_k$, with the $a_k$'s equal to $0$ or $1$. The intent here is to restrict the set of equations considered to those $k$ for which $a_k$ takes the value $1$.
The corresponding spaces are denoted $\widetilde{\mathcal{H}}_{\mbf{x}}^{\mbf{a}}$ and $\widehat{\BH}_{\mbf{x}}^{\mbf{a}}$, with similar convention for the corresponding Hilbert series and graded Frobenius characteristics. It follows that, even in this more general context, we have

\begin{coro} For all choices of $a_k$, 
\begin{equation}
   \widetilde{\mathcal H}_n^{\mbf{a}}(t)= \sum_{k=0}^n{n \choose k}t^k\,
        \widehat{\mathcal H}_k^{\mbf{a}}(t).
 \end{equation}    
\end{coro}
In particular, if conjecture~\ref{conj_deux}  holds then (\ref{conj_hilb}) holds, and this immediately implies (\ref{eq_conj_un}).
There is an even finer corollary of Theorem~\ref{thm_quatre}. 

\begin{coro}\label{prop_deux}
The graded Frobenius characteristic of $\widetilde{\mathcal{H}}^\mbf{a}_{\mbf{x}}$ is given by 
the symmetric function
 \begin{equation}\label{frobenius_star}  
    \widetilde{F}_n^\mbf{a}(t)= \sum_{k=0}^n t^k \widehat{F}_k^\mbf{a}(t)\, h_{n-k}(\mbf{z})
 \end{equation}
\end{coro}

\begin{proof}[\bf Proof.]
Let $\mbf{x}=\mbf{y}+\mbf{z}$ (disjoint union), and denote $\S_\mbf{x}$ the group of permutation of the variables $\mbf{x}$. Recall that, if $f(\mbf{y})$ is a polynomial which under the action of $\S_\mbf{y}$  generates an irreducible module with character $\chi^\lambda$ (for some $\lambda\vdash k$), then under the action of $\S_\mbf{x}$ it will generate an $\S_\mbf{x}$-module whose  character is obtained by inducing, from  $\S_\mbf{y}\times \S_\mbf{z}$ to $\S_\mbf{x}$,  the product of $\chi^\lambda$ by the trivial character
of $\S_\mbf{z}$. In particular it follows that the Frobenius characteristic of the
$\S_\mbf{x}$-module generated by $f(\mbf{y})$ will simply be $s_\lambda\,h_{n-k}$. It follows that the Frobenius characteristic of the direct sum
$$
 \bigoplus_{\mbf{y}\subseteq
 \mbf{x} } e_\mbf{y}\widehat {\mathcal{H}}_{\mbf{y}}^\mbf{a}
$$
is given by the polynomial in (\ref{frobenius_star}). 
\end{proof}
Observe that conjecture~\ref{conj_deux} implies that $\widehat{F}_k(t)$ is equal to $F_k(t)$, the graded Frobenius characteristic of the harmonics of $\S_n$ given by formula (\ref{avec_maj}) (or equivalently (\ref{regular})). 

A conjecture of Wood~\cite[conjecture 7.3]{wood} is thus partially addressed. It states that a part of the quotient $\widetilde{\mbf{C}}$ of the ring $\R[\mbf{x}]$, by the subspace of hit-polynomials for the operators $\widetilde{D}_k$, affords as a basis the (equivalence classes of)  monomials 
$\mbf{x}^\mbf{a}$, with $\mbf{a} =(a_1,\ldots,a_n)$ such that $ 1\leq a_i\leq i$. In fact, our discussion suggests that the full basis is given as the set of monomials 
$\mbf{y}^\mbf{a}$, with $\mbf{a} =(a_1,\ldots,a_k)$ such that $ 1\leq a_i\leq i$. 
Here $\mbf{y}$ varies in all $k$-subsets of $\mbf{x}$ (with induce order on variables), and $k$ goes from $0$ to $n$. Thus, for $n=3$ we would get the basis
$$\begin{array}{lll} 
  1, x_1,x_2,x_3, x_1x_2,x_1x_2^2, x_1x_3,x_1x_3^2,  x_2x_3,x_2x_3^2,\\
    x_1x_2x_3,x_1x_2^2x_3,x_1x_2x_3^2,
    x_1x_2^2x_3^2,x_1x_2x_3^3,x_1x_2^2x_3^3.
 \end{array}$$
 In view of Theorem~\ref{thm_quatre}, Wood's conjecture is a consequence of Conjecture~\ref{conj_deux} and the fact that $\widetilde{\mbf{C}}$ is isomorphic to $\widetilde{\mathcal{H}}_{\mbf{x}}$ as a graded $\S_n$-module.

\section{More on $q$-harmonics}
 We will now link the study of harmonics of the $\widetilde{D}_k$ to further our understanding of the common zeros of the operators $D_{k:q}$, in the case when $q$ is considered as a formal parameter. Our point of departure is the following important fact. Let us denote
$\nabla_k$ the generalized Laplacian
    $$\nabla_k:=\sum_{i=1}^n \partial_i^k,$$
thus we have
    $$D_{k:q}=q\widetilde{D}_k+\nabla_k.$$
Then

\begin{theo}\label{thm_un} Up to a power of $q$, every $q$-harmonic polynomial $f$  may be written  in the form
 \begin{equation}\label{harm_series}
      f= f_0+q f_1+q^2 f_2+\cdots + q^m f_m
 \end{equation}
 with $ f_i\in \R[\mbf{x}]$, and such that for all $k\ge 1$ we have
 \begin{equation}\label{harm_conds}
     \begin{array}{ll}
     {\rm (a)}\  \nabla_k f_0=0\\[8pt]
     {\rm (b)}\   \nabla_k f_i= - \widetilde{D}_kf_{i-1}, \qquad{\rm for\ all}\quad   i=2,\ldots, m-1,\\[8pt]
     {\rm (c)}\  \widetilde{D}_kf_{m}=0.
\end{array}
 \end{equation}
In particular, it follows that for any $r\ge 0$, and any choice of $k_1,k_2,\,\ldots\,,k_r\ge 1 $, the element 
 \begin{equation}\label{harm_condbis}
    \nabla_{k_1}\nabla_{k_2}\cdots\nabla_{k_r}f_r
 \end{equation}
is a  $\S_n$-harmonic polynomial in the usual sense.
\end{theo}

\begin{proof}[\bf Proof.]
Clearly (\ref{harm_series}) can be obtained by expanding the given element in powers of $q$. The identities in (\ref{harm_conds}) are then  immediately obtained by equating powers of $q$ in the identity
          $$\nabla_k f= -q\, \widetilde{D}_kf  $$
In particular we see that $f_0$ must necessarily be harmonic, and this shows the case $r=0$ of (\ref{harm_condbis}). We  can thus proceed by induction
and assume that (\ref{harm_condbis}) is harmonic  for $r-1$. This given, it follows that, for any $k_1,k_2,\ldots, k_{r }$ and $k\ge 1$, we  have
  \begin{eqnarray*}
      A\,[\nabla_{k_s},\widetilde{D}_\ell]\,B\,f_{r-1}&=&  k_s \,A\,\nabla_{\ell+k_s}\,B\,f_{r-1} \\ &=& 0
  \end{eqnarray*}
with $A$ standing for the operator  $\nabla_{k_1} \cdots\nabla_{k_{s-1}}$, and $B$ for $\nabla_{k_{s+1}} \cdots\nabla_{k_{r }}$.
This may further be rewritten as
    $$ A\,\nabla_{k_s} \widetilde{D}_\ell \,B\,f_{r-1}= A\,\widetilde{D}_\ell \nabla_{k_s} \,B\,f_{r-1} $$
which implies that
\begin{equation}\label{commute}
  \nabla_{k_1}\nabla_{k_2}  \cdots\nabla_{k_{r}}\widetilde{D}_\ell f_{r-1}= 
       \widetilde{D}_\ell\nabla_{k_1}\nabla_{k_2}\cdots \nabla_{k_{r}}f_{r-1} 
\end{equation}
But the inductive hypothesis gives
$$
\widetilde{D}_\ell\nabla_{k_1}\nabla_{k_2}\cdots  \nabla_{k_{r}}f_{r-1}= 0
$$
and thus (\ref{commute}) combined with (\ref{harm_conds}) becomes
   $$ -\nabla_{k_1}\nabla_{k_2}\cdots \nabla_{k_{r}}\nabla_\ell f_{r}= 0 $$
this completes the induction and our proof. 
\end{proof}

This brings us in a position to   recall   the following basic result of Hivert-Thi\'ery

\begin{theo}[Hivert--Thi\'ery]\label{thm_deux}
  Let $\BK$ denote the field $\R(q)$ of rational fractions in $q$, and set
 \begin{equation}\label{ideal_sym}
    {\mathcal I}:=(e_1,e_2,\,\ldots\,,e_n)_{\BK[\mbf{x} ]}  
  \end{equation}
The vectors spaces $ {\mathcal{H}}_{\mbf{x}:q} $ and ${\mathcal I}$ are complementary.
 Therefore   the
 Hilbert series of   ${\mathcal{H}}_{\mbf{x}:q} $ is a sub-polynomial of the Hilbert series of  $\S_n$-Harmonics with coefficients in $\BK$. That is
 \begin{equation}\label{sub_harm}
      \sum_{d=0}\dim \pi_d({\mathcal{H}}_{\mbf{x}:q}) t^m \ll [n]_t!,
  \end{equation}
with ``$\ll$'' denoting coefficient wise inequality.  
\end{theo}

Let us reformulate the expansion of (\ref{harm_series}) in the form 
  $$ f= q^r(f_0+q f_1+\cdots +q^m f_m)\qquad  (\hbox{with each } f_i\in \R[\mbf{x}],\  f_i\not=0) $$
We  call $f_0$  the  {\em first term} of $f$ and denote it  ``$\FT(f)$''.
Analogously we say that $f_m$ is  the {\em last term} of $f$  and  denote  it ``$\LT(f)$''.
The integer $m$ will be called the {\em length} of $f$. We also set
  \begin{equation}\label{def_BHF}
        {\mathcal{H}}_{\mbf{x}}^F := {\mathcal L}[\FT(f)\ |\  f\in {\mathcal{H}}_{\mbf{x}:q}]\quad\mathrm{and}\quad  {\mathcal{H}}_{\mbf{x}}^L := {\mathcal L}[\LT(f)\ |\  f\in {\mathcal{H}}_{\mbf{x}:q}]
  \end{equation}
  to respectively stand for the span of first terms of $q$-harmonics and last terms.
Theorem~\ref{thm_deux} has the following remarkable corollary.

\begin{theo}\label{thm_trois}
The three spaces ${\mathcal{H}}_{\mbf{x}}^F$, ${\mathcal{H}}_{\mbf{x}}^L $  and ${\mathcal{H}}_{\mbf{x};q}  $ are equivalent as graded $\S_n$-modules and therefore they are all isomorphic
to  a submodule of the Harmonics of $\S_n$. \end{theo}

\begin{proof}[\bf Proof.]
Since ${\mathcal{H}}_{\mbf{x};q}$ is an $\S_n$-module, it follows that it has a direct sum decomposition of the form
\begin{equation}\label{dim_grad}
{\mathcal{H}}_{\mbf{x};q}= \bigoplus_{\lambda}\bigoplus_{1\le i\le m_\lambda }\BM^{\lambda,i}
\end{equation}
with $\lambda$ varying in the set of partitions of $n$, and $m_\lambda$ yielding the multiplicity of the irreducible character  $\chi^\lambda$ (associated to the irreducible representation having dimension $n_\lambda$).
in the character of ${\mathcal{H}}_{\mbf{x};q}$. Moreover since  ${\mathcal{H}}_{\mbf{x};q}$ is a graded $\S_n$-module. Each of the components $\BM^{\lambda,i}$
can be chosen to consist of homogeneous elements of ${\mathcal{H}}_{\mbf{x};q}$ the same degree. This given, from elementary representation theory
it follows that each $\BM^{\lambda,i}$ is generated (under the action of the group algebra ${\mathcal A}(\S_n)$) by any one of its elements. 
Thus, we can select an element  $\varphi \in   \BM^{\lambda,i}$ from which we get a basis of $ \BM^{\lambda,i}$ in the form
    $$\{\varphi ,\theta_2\cdot \varphi ,
       \theta_3\cdot \varphi ,\ldots, 
       \theta_{n_\lambda}\cdot \varphi \}$$
for suitable elements $\theta _2,\theta _3,\,\ldots\,,\theta _{n_\lambda}\in {\mathcal A}(\S_n)$.  Moreover,  up to a re-normalization of $\varphi$ by a power of $q$, we can suppose that
$\varphi$ is of the form
$$
\varphi (\mbf{x};q)= \varphi_{0} (\mbf{x})+q\,\varphi_{1}(\mbf{x})+\cdots +q^m\varphi_{m}(\mbf{x}), 
$$
with $\varphi_{0}(\mbf{x})$ and $\varphi_{m}(\mbf{x})$ non vanishing. This will imply that
all the other elements  $ \psi\in \BM^{\lambda,i}$ will have a $q$-expansion of the form
$$
\varphi(\mbf{x};q)= \theta\cdot \varphi_{0}(\mbf{x})+q\,\theta\cdot \varphi_{1}(\mbf{x})+\cdots +q^m\theta\cdot \varphi_{m}(\mbf{x}) \qquad (\hbox{for some $\theta \in {\mathcal A}(\S_n)$})
$$
with $\theta\cdot \varphi_{0}(\mbf{x})$ and $\theta \cdot \varphi_{m}(\mbf{x})$ also non-vanishing. The reason for this is that 
the irreducibility of $ \BM^{\lambda,i}$ forces the element $\psi(\mbf{x};q)$ itself to generate  $ \BM^{\lambda,i}$ and if
any of these two components of $\theta\cdot \varphi(\mbf{x})$ where to vanish then the corresponding component of $\varphi(\mbf{x})$  would also have to vanish. This establishes explicit graded $\S_n$-modules isomorphisms between ${\mathcal{H}}_{\mbf{x};q}  $  and each of the spaces ${\mathcal{H}}_{\mbf{x}}^F$ and ${\mathcal{H}}_{\mbf{x}}^L $. Since each element in the space ${\mathcal{H}}_{\mbf{x}}^F$ is an harmonic polynomial (by virtue of (\ref{harm_conds}.a)) the
 proof is completed.
\end{proof}

In view of (\ref{dim_grad}), this result expresses the dimension of ${\mathcal{H}}_{\mbf{x}:q}$ in the form
\begin{equation}\label{dim_BHn}
   \dim {\mathcal{H}}_{\mbf{x}:q}= \sum_{\lambda\vdash n}m_\lambda n_\lambda\le \sum_{\lambda   
        \vdash n}n_\lambda^2=n!
\end{equation}
thus the single equality $\dim {\mathcal{H}}_{\mbf{x}:q}= n!$ would imply that $ {\mathcal{H}}_{\mbf{x}:q}$ affords the regular representation of $\S_n$. 
In particular this would yield that ${\mathcal{H}}_{\mbf{x}}^F$ is none other than the space of harmonics of $\S_n$. Since ${\mathcal{H}}_{\mbf{x}:q}$ is isomorphic to ${\mathcal{H}}_{\mbf{x}}^F$,
as a graded $\S_n$-module, it would follow that  ${\mathcal{H}}_{\mbf{x}:q}$  itself is isomorphic to the space of harmonics of $\S_n$ (as a graded $\S_n$-module). Thus the Hivert-Thi\'ery conjecture may simply be proved by showing that  (\ref{dim_BHn}) is in fact an equality.

\section{The Kernel of $D_k$}
The compute the general space ${\mathcal{H}}_\mbf{x}$ of harmonic polynomials,  we need to find common solutions of the differential equations
\begin{equation}\label{lequation}
   \sum_{i=1}^n a_{i,k}x_i \partial_{i}^{k+1} + b_{i,k} \partial_{i}^k f(\mbf{x})= 0.
   \end{equation} 
We begin by studying the space of solutions of just one such equation. It develops that the corresponding kernel of the operator 
$D_k$ may be given a precise explicit description whenever $a_{i,k}\,d+b_{i,k}\not=0$, for all $d\in \N$. Since we will work with a fixed value of $k$, we will lighten the notation by writing simply $a_{i}$ instead of $a_{i,k}$.

The case $k=1$ illustrates all aspects of the method. We   construct a set 
  \begin{equation}\label{la_base}
       \{\mbf{y}^\mbf{r}+\Psi_1(\mbf{y}^\mbf{r})\}_{\mbf{r}\in \N^{n-1}}
   \end{equation}
which is a basis of the solution set of (\ref{lequation}), for $k=1$. Here, $\Psi_1$  is a linear operator described below. Let us start with an adapted and reformulated observation of Hivert and Thi\'ery \cite[Proposition 18]{hivert_thiery}. Simply writing $x$ for $x_n$, and $\mbf{y}$ for the set of variables $x_1,\ldots, x_{n-1}$, we expand $f\in\R[\mbf{x}]$  as polynomials in $x$:
\begin{equation}\label{expansion_xn}
   f=\sum_d f_d  \frac{x^d}{d!},\qquad \mathrm{with}\ f_d\in\R[\mbf{y}].
\end{equation}
The effect of $D_1$ can then be described in the format
\begin{equation}\label{reduction_D}
  D_1\left(\sum_d f_d   \frac{x^d}{d!}\right)=
     \sum_d \left[D_1(f_d ) + (d\,a_n+b_n)f_{k+1} \right] \frac{x^d}{d!}.
\end{equation}
Setting $a:=a_n$ and $b:=b_n$, we now assume that $a\,d+b\not=0$, for all $d\in\N$. Then, the right-hand side of (\ref{reduction_D}) vanishes if and only we choose $f$ to be such that
\begin{equation}\label{recurrence_f}
    f_{d+1} =\frac{-1}{ a\,d+b}\, D_1(f_d ),
\end{equation}
for all $d\geq 0$. Unfolding this recurrence for the $f_d $'s, we find that
every element of the kernel of $D_1$ can be written as $f_0+\Psi_1(f_0)$, if we define the linear operator $\Psi_1$ as 
    \begin{equation}\label{formule_phi}
          \Psi_1(g) := \sum_{m\geq 1}  (-1)^m \frac{D_1^m(g)}{[a;b]_m} \,\frac{x^m}{m!}, \qquad \mathrm{for} \ g\in\R[\mbf{y}].
    \end{equation}
Here  we use the notation
       $$[a;b]_m:=b\,(a+b)\,(2\,a+b) \cdots ((m-1)\,a+b).$$ 
This leads to the following theorem.

\begin{theo}\label{thm_quatre_un}
The collection of polynomials $\mbf{y}^\mbf{r}+ \Psi_1(\mbf{y}^\mbf{r})$  is a basis 
for the kernel of $D_1$. In fact, given any polynomial $f$ in the kernel of $ D_1$,
its expansion in terms of this basis is simply obtained as
\begin{equation}\label{formule_base}
    f =\sum_{\mbf{r}}  a_{\mbf{r}} (\mbf{y}^\mbf{r}+\Psi_1(\mbf{y}^\mbf{r}))
   \end{equation}
with $(f \mod x)=\sum_{\mbf{r}}a_{\mbf{r}} \mbf{y}^\mbf{r}$.
\end{theo}

It follows readily from this theorem that
\begin{prop}\label{generic_ker}
Whenever $a\,d+b\not=0$ for all $d\in\N$, the Hilbert series of the dimension of the kernel of $D_1$ is 
   $$\frac{1}{(1-t)^{n-1}}$$
\end{prop}

In view of Theorem~\ref{thm_quatre}, it follows that the Hilbert series of the kernel of $\widetilde{D}_1$ is 
\begin{equation}\label{borne_ker_star}
     1+\sum_{k=1}^n {n \choose k} t^k\frac{1}{(1-t)^{k-1}}.
\end{equation}
In fact, we can get an explicit description of this kernel using (\ref{eq_prop_un}).  

We can generalize formula (\ref{formule_phi}) to get a description of the kernel of $D_k$ as follows.
Observe as before that
\begin{equation}\label{reduction_gen}
  D_k\left(\sum_d f_d  \frac{x^d}{d!}\right) =
   \sum_d \left[D_k(f_d ) +(a\, d+b)f_{d+k}\right] \frac{x^{d}}{d!}.
\end{equation}
For this expression to be zero, we must have
     $$f_{d+k} =\frac{-1}{ a \,d+b}\, D_k(f_{d}),$$
 with the same conditions as before on $a$ and $b$.
This recurrence has a unique solution given initial values for $f_d$, $0\le d\le k-1$. Clearly these can be fixed at leisure. 
Substituting the solution of the recurrence in $f$,
we get an element of the kernel of $D_k$ if and only if $f$ is of the form
    $$f=(f \mod x^k) + \Psi_k(f \mod x^k),$$
with $\Psi_k$ the linear operator defined as
    \begin{equation}\label{formule_phi_ell}
         \Psi_k\left(\sum_{r=0}^{k-1} f_r  \frac{x^r}{r!}\right):= \sum_{m\ge 1} \sum_{r=0}^{k-1}(-1)^m \frac{D_k^m(f_r)}{[a\,k;a\,r+b]_m} \,\frac{x^{k\,m+r}}{(k\,m+r)!}.
    \end{equation}
In particular, it follows that the Hilbert series of the kernel of $D_k$ is
\begin{equation}
   (1+t+\ldots t^{k-1})\,\frac{1}{(1-t)^{n-1}}.
\end{equation}

\section{Some explicit harmonic polynomials}
Common zeros of all $D_k$'s are exactly what we are looking for. Some of these are easy to find when the $D_k$'s are symmetric. This condition holds for the special cases considered in sections~\ref{sectionq}  and \ref{sectiontilde}. 

Now, let $\lambda$ be any partition of $n$, and consider a tableau $\tau$ of shape $\lambda$, this is to say a bijection
   $$\tau:\lambda\longrightarrow \{1,2,\ldots, n\},$$
 with $\lambda$ identified with the set of cells of its Ferrers diagram.
 Recall that, for $\lambda=\lambda_1\geq \lambda_2\geq \ldots\geq \lambda_k>0$,  the {\em cells} of $\lambda$ are the $n$ pairs $(i,j)$ in $\N^2$, such that
     $$1\leq i\leq  \lambda_j,\qquad 1\leq j\leq k.$$
 The value $\tau(i,j)$ is called an {\em entry} of $\tau$, and it is said to lie in {\em column} $i$ of $\tau$. 
The {\em Garnir polynomial} of a $\lambda$-shape tableau $\tau$, is defined to be
   $$\Delta_\tau(\mbf{x}):=\prod_{i,\ j<k}
       (x_{\tau(i,j)}-x_{\tau(i,k)}).$$
In other terms, the factors that appear in $\Delta_\tau(\mbf{x})$ are differences of entries of $t$ that lie in the same column.

Now, define $\mathcal{V}_\lambda$ to be the linear span  of the polynomials $\Delta_\tau$, for $\tau$ varying in the set of tableaux of shape $\lambda$. In formula,
    $$\mathcal{V}_\lambda:=\R[ \Delta_\tau\ |\ \tau\ \mathrm{tableau\ of\ shape}\ \lambda\}.$$
It is well known that  this homogeneous (invariant) subspace is an irreducible representation of of $\S_n$ of dimension equal to the number of standard Young tableaux. Moreover, in the ring $\R[\mbf{x}]$, there exists no isomorphic copy of this irreducible representation lying in some homogeneous component of degree lower then that in which lies $\mathcal{V}_\lambda$. It is easy to check that the degrees of all of the $\Delta_\tau$'s, for a tableau of shape $\lambda$, are all equal to
   $$\sum_{i=1}^{\ell(\lambda)} (i-1)\, \lambda_i,$$
 which is usualy denoted $n(\lambda)$ in the literature (See~\cite{macdonald}). This is the smallest possible value for the  cocharge of a standard tableau of shape $\lambda$.
This fact has the following easy implication.

\begin{prop} For any tableau $\tau$ of shape $\lambda$, the Garnir polynomial $\Delta_\tau(\mbf{x})$ is a zero of $D_k$, for $k\geq 1$, whenever $D_k$ is symmetric.
\end{prop} 

\begin{proof}[\bf Proof.] We need only observe that $D_k$ is a degree lowering morphism of representation. Indeed, in view of Schur's lemma, either $D_k(\mathcal{V}_\lambda)$ is isomorphic   to $\mathcal{V}_\lambda$ or it is simply reduced to the trivial subspace $\{0\}$. In view of the minimality of the degree of $\mathcal{V}_\lambda$, only the later possibility can occur, and we are done with the proof.
\end{proof}

A direct consequence of this is that there is at least one copy of each irreducible representation of $\S_n$ in ${\mathcal{H}}_{\mbf{x}}$, when the $D_k$' s are all symmetric. Moreover, under the same conditions, we have
    $$\sum_{\lambda\vdash n} f_\lambda\,t^{n(\lambda)}\ll
        H_n(t), $$
with ``$\ll$'' denoting coefficient wise inequality.

\section{More dimension estimates}
To explore further possible dimension estimates, we now make use of the dual operators $D_k^*$, assuming that condition (\ref{condition}) holds all through this section. Recall that this is true for the special cases of sections~\ref{sectionq} and \ref{sectiontilde}. 

\begin{prop}\label{prop_base}
 Let ${\mathcal B}_{n}$  be a homogeneous basis:
    $${\mathcal B}_{n} =\bigcup_{d\geq 0}{\mathcal B}_{d,n},$$
 with $\mathcal{B}_{d,n} =\pi_d(\mathcal{B}_n)$,
 for the space ${\mathcal{H}}_{\mbf{x}}$, with the $D^*_k$'s  satisfying (\ref{condition}). Then every homogeneous degree  $m$ polynomial $f(\mbf{x})$ has an expansion of the form
\begin{equation}\label{expansion}
f(\mbf{x})=\sum_{d=0}^m\sum_{g\in{\mathcal B}_{d,n}}\sum_{\lambda\vdash m-d}c_{\lambda,g}D^*_\lambda\,  g(\mbf{x})  
\end{equation}
\end{prop}

\begin{proof}[\bf Proof.]
We proceed by induction on the degree $m$ of $f$, observing that the degree $0$ case holds trivially.
It follows readily from definitions that  the space ${\mathcal{H}}_{\mbf{x}}$ is the orthogonal complement of the graded vector space
$$
{\mathcal L}\Big[\sum_{k\geq 1} D^*_k\,g_k\ |\ g_k\in \R[\mbf{x}]\Big]
$$
From this observation, it follows that every degree $d$ homogeneous polynomial 
$f$ has an expansion of the form
$$
f(\mbf{x})=\sum_{g\in{\mathcal B}_{m,n}}c_{g}\, g(\mbf{x}) +\sum_{k\geq 1}  D^*_k\,f_k(\mbf{x})$$
with the $f_k(\mbf{x})$'s homogeneous polynomials of degree $m-k$. Applying the induction hypothesis to 
each of the $f_k$'s, we get
$$
f(\mbf{x})=\sum_{d=0}^m\sum_{g\in{\mathcal B}_{d,n}}\sum_{\alpha\models m-d}c_{\alpha,g}D^*_\alpha\,  g(\mbf{x})\, .
$$
Proposition~\ref{red_comp} guarantees
that the composition indexed operators $D_\alpha^*$ in such an expansion can be converted into partition indexed operators $D_\lambda^*$, completing the proof.
\end{proof}

As a corollary we derive the following remarkable identity for the Hilbert series of $q$-harmonics.

\begin{theo}\label{thminegalite}
Let $\mathcal{B}_n$ be a homogeneous basis, as in proposition~\ref{prop_base}, of the space ${\mathcal{H}}_{\mbf{x};q}$ of $q$-harmonics. Then,
denoting by $c_{d,n}$ the cardinality of ${\mathcal B}_{d,n}$,  we have 
\begin{equation}\label{borne}
 c_{d,n}= [n]_t!\Big|_{t^d},
\end{equation}
for all $d\le n$, with the right-hand side denoting the coefficient of $t^d$ in $[n]!_t$.
\end{theo}

\begin{proof}[\bf Proof.]
Let $P_k$ stands for the number of partitions of $k$, with parts of size at most $n$.
Recall that the generating series for these numbers is
   $$\sum_k {P_k} t^k = \prod_{k\geq 1} \frac{1}{1-t^k}.$$
 From the expansion in (\ref{expansion}) it follows that
\begin{equation}\label{inegalite}
\frac{1}{ (1-t)^n}\ \ll\ \sum_{m\ge 0} \sum_{d\ge 0} c_{d,n}P_{m-d}\,t^m\ = 
\sum_{d\ge 0} c_{d,n} t^d
\prod_{k\ge 1}\frac{1}{1-t^k},
\end{equation}
On the other hand, for the Hilbert series of ${\mathcal{H}}_{\mbf{x};q}$:
$$
{\mathcal H}_{\mbf{x};q}(t)=\sum_{d\ge 0} c_{d,n}\, t^{d},
$$
we have (See (\ref{sub_harm})) that
$$
\sum_{d\ge 0} c_{d,n}\, t^{d}\ll  [n]_t!.
$$
Multiplying both sides of this by $\prod_i (1-t^i)^{-1}$, and using  (\ref{inegalite}), we derive that
\begin{eqnarray}
\frac{1}{ (1-t)^n} &\ll&
 \sum_{d\geq 0}^m c_{d,n}
\prod_{i\ge 1}\frac{1}{1-t^i} \nonumber\\
&\ll& 
\sum_{d\ge 0} [n]_t! 
\prod_{i\ge 1}\frac{1}{1-t^i} \label{inegalite_bis}
\end{eqnarray}
However the simple identity
$$
\frac{1}{ (1-t)^n}=[n]_t!\prod_{i=1}^n\frac{1}{1-t^i}
$$
forces all inequalities between coefficients  to be equalities, when $d$ is between $0$ and $n$. This completes our proof.
\end{proof}

A refined reading of this proof reveals that we would actually have equality for all $d$'s, if we could show that elements of ${\mathcal B}_{n}$ could be chosen so that the all variables appear with power less or equal to $n-1$. Indeed this would imply that indices of the $D^*_\lambda$'s, appearing in (\ref{expansion}), would involve parts of size at most $n$. Then  (\ref{inegalite}) could be written in the stronger form\footnote{Observe that we have here a finite product rather the an infinite one.}
\begin{equation}
\frac{1}{ (1-t)^n} \ \ll\ \sum_{m\ge 0} \sum_{d\ge 0} c_{d,n}\Pi_{m-d}\,t^m\ \ll\ 
\sum_{d\ge 0} c_{d,n} t^d
\prod_{i= 1}^n\frac{1}{1-t^i}.
\end{equation}
Continuing the argument as in the proof would then lead to the conclusion that these three series actually coincide.
 
 \section{A new regular sequence and a universal dimension bound}\label{sec_regular}

 The goal of this section is to establish a bound for the dimension of $ {\mathcal{H}}_{\mbf{x}:q}$ which is valid
for all values of $q$. To carry this out we need some auxiliary results from  commutative algebra. 
Let $\CF$ be an algebraically closed field and let $\theta_1(\mbf{x}),\theta_2(\mbf{x}) ,\,\ldots\,,\theta_n(\mbf{x})$ be homogeneous polynomials of  $\CF[\mbf{x}]$ of respective degrees $d_1,d_2,\,\ldots\,,d_n$. The following result is basic.

\begin{prop}\label{reg_theta}
The polynomials $\theta_1(\mbf{x}),\theta_2(\mbf{x}) ,\,\ldots\,,\theta_n(\mbf{x})$ form a regular sequence in  $\CF[\mbf{x}]$ if and only if
the system of equations
$$
\theta_1(\mbf{x})=0\, ,\ \theta_2(\mbf{x})=0 \, ,\, \ldots\, ,\ \theta_n(\mbf{x})=0
$$
has,  for $\mbf{x}\in \CF^n$,   the unique solution 
$$
x_1=0\, ,\ x_2=0\, ,\ \ldots  \, ,\ x_n=0. 
$$
\end{prop}

\begin{proof}[\bf Proof.]
 Let us denote $\Theta_n$ the ideal $(\theta_1,\theta_2,\,\ldots\,,\theta_n)_{\CF[\mbf{x}]}$ generated by the $\theta_k$'s in $\CF[\mbf{x}]$.
If  $\theta_1(\mbf{x}),\theta_2(\mbf{x}) ,\,\ldots\,,\theta_n(\mbf{x})$ is a regular sequence, and $\{m_1,m_2,\ldots, m_N\}$
is a monomial basis for the quotient
\begin{equation}\label{quotient_delta}
\CF[\mbf{x}]/\Theta_n,
\end{equation}
then an homogeneous polynomial $f(\mbf{x})\in \CF[\mbf{x}]$, of degree $d$, has a unique expansion of the form
\begin{equation}\label{develop_theta}
f(\mbf{x})=\sum_{i=1}^N\sum_{\sum_k d_k\,r_k=d-\deg(m_i)}
 c_{i;\mbf{r}}m_i(\mbf{x})\theta_1^{r_1}(\mbf{x})\theta_2^{r_2}(\mbf{x})\cdots \theta_n^{r_n}(\mbf{x}),
\end{equation}
with $c_{i;\mbf{r}}\in \CF$, and $\mbf{r}\in \N^n$. In particular it follows from this that the Hilbert series of the quotient
in (\ref{quotient_delta}) is given by the polynomial 
\begin{equation}\label{hilbert_theta}
\sum_{i=1}^N t^{\deg(m_i)}=\prod_{i=1}^n\frac{ 1-t^{d_i} }{ 1-t  }=\prod_{i=1}^n(1+t+\cdots +t^{d_i-1} ).
\end{equation}
Thus 
$$
N=\dim \CF[\mbf{x}]/\Theta_n=d_1d_2\cdots d_n.
$$
In addition,  from (\ref{hilbert_theta}) we deduce  that 
$$
d_{\max}:=d_1-1+d_2-1+\cdots +d_n-1= \max_{1\le i\le N} \deg(m_i)
$$
We see from this, and the expansion in (\ref{develop_theta}), that any homogeneous  polynomial of degree $d>d_{\max}$
will necessarily be in the ideal $\Theta_n$. In particular for
any $1\le i \le n$ we will have polynomials $A_{i;1},A_{i;2},\,\ldots\,,A_{i;n}\in \CF[\mbf{x}]$ giving
\begin{equation}\label{decomposition_theta}
x_i^{d_{\max}{+1}}=A_{i;1}(\mbf{x})\theta_1(\mbf{x})+A_{i;2}(\mbf{x})\theta_2(\mbf{x})+\ldots +A_{i;n}(\mbf{x})\theta_n(\mbf{x})
\end{equation}
Now, if for some $\mbf{x}'=(x_1',x_2', \ldots ,x_n')\in\CF^n $
we  have 
\begin{equation}\label{autre_sol}
\theta_1(\mbf{x}')=0\, ,\ \theta_2(\mbf{x}')=0 \, ,\, \ldots\, ,\ \theta_n(\mbf{x}')=0,
\end{equation}
then (\ref{decomposition_theta}) immediately gives that $(x_i')^{d_{\max}+1}=0$, hence
\begin{equation}\label{composantes_nulles}
 x_1'=0,\ x_2'=0, \ldots ,\ x_n'=0. 
\end{equation}
This proves   ``necessity''.

To prove the converse,  suppose that (\ref{autre_sol}) implies (\ref{composantes_nulles}). Then   the Hilbert Nullstellensatz gives that for
some exponents $N_1,N_2,\,\ldots\,,N_n$ we must have 
$$
x_i^{N_i}\in \Theta_n.
$$
In particular this implies that the quotient in (\ref{quotient_delta}) is finite dimensional and therefore 
$\theta_1(\mbf{x}),\theta_2(\mbf{x}) ,\,\ldots\,,\theta_n(\mbf{x})$ is a system of parameters. The Cohen-Macaulay-ness of $\CF[\mbf{x}]$
then gives that $\theta_1(\mbf{x}),\theta_2(\mbf{x}) ,\,\ldots\,,\theta_n(\mbf{x})$ is a regular sequence.
This completes our proof.
\end{proof}

We next make use of  this proposition to study the sequence of polynomials
$$
\varphi_m(\mbf{x}):=\sum_{i=1}^na_ix_i^m,
$$
for $m\ge 0$.
More precisely we seek to
obtain conditions on the coefficient sequence
\begin{equation}\label{cond_a}
\mbf{a}=(a_1,a_2,\,\ldots\,,a_n)\in \CF^n
\end{equation}
which assure that, for a given $k\ge1$, that the polynomials
$$
\varphi_k(\mbf{x}),\  \varphi_{k+1}(\mbf{x}),\  \ldots ,\  \varphi_{k+n-1}(\mbf{x})
$$
form a regular sequence in $\CF[\mbf{x}]$.  

We first observe that the polynomials $\varphi_m(\mbf{x})$, for $m>n$, may be expressed in term of the $\varphi_k(\mbf{x})$'s, for $1\leq k\leq n$. Indeed, recall that the ordinary elementary symmetric functions $e_r(\mbf{x})$ may be presented in the form of the identity
$$
(t-x_1)(t-x_2)\cdots (t-x_n)=\sum_{r=0} ^n  (-1)^{ r}e_{ r}(\mbf{x})\,t^{n-r}.$$
Setting $t=x_i$, we obtain 
  $$\sum_{r=0} ^n  (-1)^{ r}e_{ r}(\mbf{x})\,x_i^{n-r}=0.$$
 Multiplying both sides by $a_i\,x_i^{m-n}$ and isolating $a_i\,x_i^m$, we get
$$
a_i\,x_i^{ m}=-\sum_{r=1}^{n } (-1)^{ r}e_{ r}(\mbf{x})\, a_i\,x_i^{m-r}.
$$ 
Thus, summing up on $i$, the following recurrence results
\begin{equation}\label{form_red_phi}
\varphi_{m}(\mbf{x})=\sum_{r=1}^{n } (-1)^{r+1}e_{r}(\mbf{x})\,\varphi_{m-r}(\mbf{x}).
\end{equation}
Unfolding this recurrence, we conclude that $\varphi_m$ lies in the ideal  $(\varphi_1,\varphi_2,\,\ldots\,,\varphi_{n})_{\CF[\mbf{x}]}$, for all $m$.
 
\begin{Remark}\label{necessaire} It is interesting to observe that identity (\ref{form_red_phi}) yields that 
\begin{equation}\label{reg_seq}
    \varphi_1(\mbf{x}),\ \varphi_2(\mbf{x}),\, \ldots \,,\ \varphi_n(\mbf{x})
\end{equation}
   is never a regular sequence
when $a_1+a_2+\cdots + a_n=0$. Indeed, setting $m=n$ in (\ref{form_red_phi}), we get
$$
\varphi_m(\mbf{x})=\sum_{r=1}^{n-1 }\varphi_{m-r}(\mbf{x}) (-1)^{r+1}e_{r}(\mbf{x}) +(-1)^{n+1}e_n(\mbf{x})\big(a_1+a_2+\cdots + a_n\big)
$$
and thus the vanishing of $ a_1+a_2+\cdots + a_n$ forces $\varphi_n(\mbf{x})$ to vanish modulo the ideal
 $$(\varphi_1,\varphi_2,\,\ldots\,,\varphi_{n-1})_{\CF[\mbf{x}]}.$$
\vskip-20pt\hfill $\blacksquare$
\end{Remark}

Let us now denote
    $$\Phi_n^k:=(\varphi_k,\varphi_{k+1},\,\ldots\,,\varphi_{k+n-1})_{\CF[\mbf{x}]},$$
  the ideal in $\CF[\mbf{x}]$ generated by the $n$ polynomials $\varphi_\ell(\mbf{x})$, with $k\le \ell\le k+n-1$. We also write $\Phi_n$ for $\Phi_n^1$.
Proposition~\ref{reg_theta} and (\ref{form_red_phi}) combine to yield the following remarkable result.

\begin{theo}\label{reduction_1n}
For any $k\ge 1$ the sequence
\begin{equation}\label{haut_phi}
\varphi_{k}(\mbf{x}),\ \varphi_{k+1}(\mbf{x}),\,\ldots\,,\ \varphi_{k+n-1}(\mbf{x}),
\end{equation}
is regular if and only if the sequence
\begin{equation}\label{bas_phi}
\varphi_{1}(\mbf{x}),\ \varphi_{2}(\mbf{x}),\,\ldots\,,\ \varphi_{n}(\mbf{x}),
\end{equation}
is regular.
\end{theo}

\begin{proof}[\bf Proof.]
Identity (\ref{form_red_phi}) yields that each element of the sequence in (\ref{haut_phi}) is in the ideal $\Phi_n$. 
Thus, if $ \varphi_{k}(\mbf{x}),\varphi_{k+1}(\mbf{x}),\,\ldots\,,\varphi_{k+n-1}(\mbf{x})$ is regular, it follows from Proposition \ref{reg_theta} that the equalities
\begin{equation}\label{haut_egalite}
\varphi_{k}(\mbf{x}')=0,\varphi_{k+1}(\mbf{x}')=0,\,\ldots\,,\varphi_{k+n-1}(\mbf{x}')=0,
\end{equation}
with $\mbf{x}'\in\CF^n $,   force
$$
x_1'=0,\  x_2'=0,\ \ldots ,x_n'=0.   
$$ 
But the containement
$$
\{\varphi_{k}(\mbf{x}),\ \varphi_{k+1}(\mbf{x}),\,\ldots\, ,\ \varphi_{k+n-1}(\mbf{x})\}\subset \Phi_n
$$
yields that the equalities 
$$
\varphi_{1}(\mbf{x}')=0,\ \varphi_{2}(\mbf{x}')=0,\,\ldots\, ,\ \varphi_{n}(\mbf{x}')=0,
$$
force the equalities in (\ref{haut_egalite}). Hence the regularity of the sequence 
    $$\varphi_{1}(\mbf{x}),\ \varphi_{2}(\mbf{x}),\,\ldots\, ,\ \varphi_{n}(\mbf{x})$$
follows from (\ref{form_red_phi}).

Let us now show the converse. To begin note that the replacement $x_i\mapsto x_i^k$ and identity (\ref{form_red_phi}) gives 
\begin{equation}\label{inclusion_phi}
\varphi_i(x_1^k,x_2^k,\,\ldots\,,x_n^k)=\varphi_{ki}(\mbf{x})\in \Phi_n^k,
\end{equation}
for $1\le i\le n$.
Now, by Proposition (\ref{reg_theta}),  the regularity of $\varphi_1,\varphi_2,\,\ldots\,,\varphi_n$ yields that the equalities
\begin{equation}\label{bas_phi_xk}
\varphi_1(x_1^k,x_2^k,\,\ldots\,,x_n^k)=0,\ \varphi_2(x_1^k,x_2^k,\,\ldots\,,x_n^k)=0,\, \ldots\ ,\varphi_n(x_1^k,x_2^k,\,\ldots\,,x_n^k)=0
\end{equation}
in  $\CF^n$,  force the equalities
$$
x_1^k=0,\ x_2^k=0,\ \ldots ,x_n^k=0. 
$$
Hence we  also have
\begin{equation}\label{bas_x}
x_1 =0,\ x_2 =0,\ \ldots ,x_n =0.
\end{equation}
On the other hand, the containments in (\ref{inclusion_phi}) yield that the equalities 
$$
\varphi_k(x_1^k,x_2^k,\,\ldots\,,x_n^k)=0,\ \varphi_{k+1}(x_1^k,x_2^k,\,\ldots\,,x_n^k)=0,\, \ldots\ ,\varphi_{k+n-1}(x_1^k,x_2^k,\,\ldots\,,x_n^k)=0
$$
force the equalities in (\ref{bas_phi_xk}) and those in turn, as we have observed, force the equalities in (\ref{bas_x}). In summary,
using again (\ref{form_red_phi}), we can conclude that the regularity of $\varphi_1,\varphi_2,\,\ldots\,,\varphi_n$ forces the regularity of
$\varphi_k,\varphi_{k+1},\,\ldots\,,\varphi_{n}$ completing the proof of the theorem.
\end{proof}

This given, here and after we need only be concerned with finding conditions on $a_1,a_2,\,\ldots\,,a_n$ that assure the 
regularity of sequence $\varphi_1,\varphi_2,\,\ldots\,,\varphi_n$. 
The following result offers a useful criterion.

\begin{theo}\label{critere}
   In the ring $\CF[\mbf{x}]$, the polynomials
$$
\varphi_1,\varphi_2,\,\ldots\,,\varphi_n
$$
form a regular sequence if and only if  we have
\begin{equation}\label{grosse_borne}
x_i^{{n\choose 2}+1}\in \Phi_n.
\end{equation} 
When this happens we have the Hilbert series equalities
\begin{equation}\label{hilbert_phi}
F_{\CF[x]/\Phi_n^k}(t)=[k]_t[k+1]_t\cdots [k+n-1]_t  
\end{equation}
and, in particular,
$$
\dim \CF [x]/\Phi_n^k=(k)(k+1)\cdots (k+n-1).
$$
\end{theo}

\begin{proof}[\bf Proof.]
Since $\deg(\varphi_i)=i$ we can use the arguments in the proof of Proposition~\ref{reg_theta} with $d_i=i$ and $d_{\max}={n\choose 2}$ 
and derive that the regularity of the sequence 
    $$\varphi_1,\ \varphi_2,\,\ldots\,,\ \varphi_n$$ 
    implies  (\ref{grosse_borne}). Conversely, again following  the proof of Proposition~\ref{reg_theta}, we see  that  (\ref{grosse_borne}) in turn forces the regularity of   $\varphi_1,\varphi_2,\,\ldots\,,\varphi_n$. The remaining part of the assertion  follows
from Theorem~(\ref{reduction_1n}) and the various identities derived in the proof of Proposition~\ref{reg_theta}.  
\end{proof}

Going along the lines of Remark~\ref{necessaire}, we are now ready to  prove the following characterization of the $a_i$'s for which we have regularity.

\begin{theo}
For $k>1$, the sequences
\begin{equation}\label{reg_thm}
    \varphi_{k},\varphi_{k+1},\,\ldots\,,\varphi_{k+n-1}
 \end{equation} 
 is regular if and only if we have
 \begin{equation}\label{conditions_ai}
a_{i_1}+a_{i_2}+\cdots +a_{i_k}\not= 0,
\end{equation}
for all $1\le i_1<i_2<\cdots <i_k\le n$.
\end{theo}
 
 \begin{proof}[\bf Proof.] Theorem~\ref{reduction_1n} shows that we need only study the case $k=1$.  Moreover, we have seen that the non regularity of 
    $$\varphi_{1},\varphi_{2},\,\ldots\,,\varphi_{n}$$ 
is equivalent to the existence of a non trivial solution, in the $x_i$'s, of the system
   \begin{equation}\label{matrix_cond}
          \begin{pmatrix}
                x_1  & x_2 & \ldots & x_n \\
                x_1^{2} & x_2^{2}& \ldots & x_n^{2}\\
                \vdots & \vdots & \ddots & \vdots\\
                x_1^{n} & x_2^{n}& \ldots & x_n^{n}
          \end{pmatrix}
          \begin{pmatrix} a_1\\ a_2\\ \vdots\\ a_n \end{pmatrix}=\begin{pmatrix} 0\\ 0\\ \vdots\\ 0\end{pmatrix}
   \end{equation}
Now, suppose that for some non empty subset $K$ of $\{1,2,\ldots,n\}$ we have $\sum_{k\in K} a_k=0$. Then setting all $x_k$, for $k$ in $K$, equal to the same nonzero value, we evidently get a non trivial solution of (\ref{matrix_cond}). Thus we see that the existence of a relation such as (\ref{conditions_ai}) implies non regularity of the sequences (\ref{reg_thm}).

To see the reverse implication, we proceed by induction on $n$, the case $n=1$ being evident. Suppose that we have $a_i$'s for which (\ref{reg_thm}) is not regular. We wish to show the existence of a relation of form (\ref{conditions_ai}), with $k>0$. We can assume that all  $a_i$'s are nonzero, since otherwise we are done. Non regularity implies that we have a non trivial solution of (\ref{matrix_cond}), in the $a_i$'s. This  forces the determinant  
\begin{equation}\label{det_thm}
 \det(x_j^i)_{1\leq i,j\leq n}=   x_1 x_2\cdots x_n\, \prod_{i>j} (x_i-x_j),
\end{equation}
to vanish.
Hence must have one of the $x_i$ equal to 0, or  $x_i=x_j$ for some $i\not=j$. In the first of these cases, we may assume that $i=n$ without loss of generality. We are reduced to the $n-1$ situation by ``restriction'' to the hyperplane $x_i=0$, and considering the sequence
 $\varphi_i(\mbf{y})$, with $1\leq i \leq n-1$,  in the $n-1$ variables $\mbf{y}$, which we get by removing from $\mbf{x}$ the variable $x_i$.
 Likewise, when $x_i=x_j$, we restrict to the corresponding hyperplane, considering the operators
 $$(a_i+a_j)x_i^m+\sum_{k\not =i,\ k\not=j} a_k\,x_k^m,$$  
for which we use the induction hypothesis. This completes the proof.
\end{proof}

We intend to derive the consequences of this assumption  in the theory of $q$-harmonics. First, we simply reformulated every statement modulo the substitution of variables 
\begin{eqnarray*}
    (a_1,a_2,\ldots,a_n)&\mapsto& \mbf{x}=(x_1,x_2,\ldots,x_n),\\
    \mbf{x}&\mapsto& \xi=(\xi_1,\xi_2,\,\ldots\,,\xi_n),
 \end{eqnarray*}
and we now have
  $$\Phi_n^k=(\varphi_k(\xi),\varphi_{k+1}(\xi),\,\ldots\,,\varphi_{k+n-1}(\xi))_{\CF_\mbf{x}[\xi]}.$$
This given, from Theorem~(\ref{critere}) we can derive the following facts about the ring 
   $$\CF_\mbf{x}[\xi_1,\xi_2,\,\ldots\,,\xi_n],$$
where now, $\CF_\mbf{x}$  denotes the field of rational functions in $\mbf{x}$ with   coefficients in $\CF$.

\begin{theo}\label{developpement_phi_nk}
Let 
$$
u_1(\xi),\, u_2(\xi),\, \ldots ,\, u_{(n+1)!}(\xi)
$$
be  a monomial basis for the quotient
$$
\CF_\mbf{x}[\xi]/\Phi_n^k,
$$
and let $\deg(u_i)=d_i$.
Then every polynomial $f(\xi)\in \CF_\mbf{x}[\xi]$, which is homogeneous of degree $d$, has a unique expansion of the form
\begin{equation}\label{form_dev_phi_nk}
f(\xi)=\sum_{i=1}^{(n+1)!}u_i(\xi) \sum_{\sum_k r_k(k+1)=d-d_i } a_{i;\mbf{r}}(\mbf{x})\varphi_1^{r_1}(\xi),\varphi_2^{r_2}(\xi)\cdots \varphi_n^{r_n}(\xi),
\end{equation}
where the coefficients $a_{i;\mbf{r}}(\mbf{x})$  are rational functions of $\mbf{x}$, for $\mbf{r}\in \N^n$. In particular if $d>{n+1\choose 2}$
then
\begin{equation}\label{conclusion_thm}
f(\xi)\equiv 0 \mod \Phi_n^k.
\end{equation}
\end{theo}

\begin{proof}[\bf Proof.]
Modulo a replacement of variables, we derive from Theorem~(\ref{critere}) that the quotient 
$\CF_\mbf{x}[\xi]/\Phi_n $ has a monomial basis consisting of $(n+1)!$ monomials.
This given, once such a basis is chosen, the statement concerning the  expansion in (\ref{form_dev_phi_nk}) is a standard  result from commutative algebra.
Finally, from Identity (\ref{hilbert_phi}) it follows that
$$
\sum_{i=1 }^{(n+1)!}t^{d_i}=[n+1]_t!
$$
 Thus  each of these monomials has degree at most $n+1\choose 2$.
We then see from (\ref{form_dev_phi_nk}) that if $f(\mbf{x})$ has degree $>{n+1\choose 2}$, then the coefficient of each $u_i$ must contain one of the
$\varphi_i$. This implies (\ref{conclusion_thm}).  
\end{proof}

Let us now denote by $\CD(\mbf{x})$ the algebra of differential operators with coefficients in $\CF_\mbf{x}$. Moreover, let
$\CD_d(\mbf{x})$ denote the subspace of $\CD(\mbf{x})$ consisting of operators of order $d$. More precisely we have
  $D\in\CD_d(\mbf{x})$   if and only if $D$ may be expanded in the form
\begin{equation}\label{def_D_Dd}
D= \sum_{|\mbf{r}|\le d}a_{\mbf{r}}(\mbf{x})\, \partial_\mbf{x}^\mbf{r}
\end{equation}
with coefficients $ a_{\mbf{r}}(\mbf{x})\in \CF_\mbf{x} $ such that  $a_{\mbf{r}}(\mbf{x})\neq 0 $     at least once when  $|\mbf{r}|= d$. We are here extending our vectorial notation to operators, so that
    $$ \partial_\mbf{x}^\mbf{r} =\partial_{1}^{r_1}\partial_{2}^{r_2}\cdots \partial_{n}^{r_n}$$
is an operator of order $|\mbf{r}|=r_1+r_2+\ldots+r_n$.
The degree condition in (\ref{def_D_Dd}) imply that the polynomial 
$$
\sigma(D):=\sum_{|\mbf{r}|= d}a_{\mbf{r}}(\mbf{x})\, \xi^\mbf{r}. 
$$
does not identically vanish. We will refer to $\sigma(D)$ as the  ``{\em symbol}'' of  $D$.

This given, as a corollary of Theorem~(\ref{critere}), we obtain the following basic result for Steenrod operators

\begin{theo}\label{operator_expansion}
Every operator $D\in\CD_d(\mbf{x})$ has an  expansion of the form   
$$
D= \sum_{i=1}^{(n+1)!}\sum_{\sum_\ell r_k(k+1)\le d-d_i}a_{i;\mbf{r}}(\mbf{x}) u_i(\partial_\mbf{x})
D_{1;q}^{r_1} D_{2;q}^{r_2}\cdots  D_{n;q}^{r_n}
$$
where $d_i=\deg(u_i)$ and  $a_{i;\mbf{r}}(\mbf{x})\in \CF_\mbf{x}$.
Note that this holds true for  any    rational value of $q$.
\end{theo}
 
\begin{proof}[\bf Proof.]
We can proceed by induction on $d$.  For $d=0$ there is nothing to prove. Thus let us assume that we have extablished
the result up to $d-1$. Assume then that $D$ is as in (\ref{def_D_Dd}) and note that from Theorem \ref{developpement_phi_nk} it follows that
we may write
$$
\sigma(D)=
\sum_{i=1}^{(n+1)!} \sum_{\sum_k r_k(k+1)=d-d_i}  
 a_{i;\mbf{r}}(\mbf{x})\,u_i(\xi)\, \varphi_1^{r_1}(\xi)\varphi_2^{r_2}(\xi)\cdots \varphi_n^{r_n}(\xi).
$$
 Now, note that if we set
$$
T:=  
\sum_{i=1}^{(n+1)!} \sum_{\sum_k r_k(k+1)=d-d_i}  a_{i;\mbf{r}}(\mbf{x})\  u_i(\partial_\mbf{x} )
D_{1;q}^{r_1} D_{2;q}^{r_2} \cdots D_{n;q} ^{r_n} ,
$$ 
then we will necessarily have the symbol equality
$$
\sigma(T)=\sigma(D),
$$ 
from which it immediately follows that the operator $D-T$ must have order stritly less than $d$.
But then the inductive hypothesis yields that $D-T$ has the desired expansion and consequently the same must hold true 
for $D$ as well. This completes the induction and the proof.
\end{proof}

We may now establish the main goal of this section.
\begin{theo}
For any value of $q$ the dimension of the space of q-Harmonic polynomials in $\mbf{x}$ does not exceed $(n+1)!$
\end{theo}

\begin{proof}[\bf Proof.]
Our goal is to derive the result by showing that there is a point $\mbf{x}'$ such that  any polynomial $f\in {\mathcal{H}}_{\mbf{x}:q}$ is uniquely determined
by the  value of the $(n+1)!$ derivatives
$$
u_1(\partial_\mbf{x})f(\mbf{x})\Big|_{\mbf{x}=\mbf{x}'},\ 
 u_2(\partial_\mbf{x})f(\mbf{x})\Big|_{\mbf{x}=\mbf{x}'},\  
\ldots, u_{(n+1)!}(\partial_\mbf{x})f(\mbf{x})\Big|_{\mbf{x}=\mbf{x}'}.  
$$
To this end we apply Theorem \ref{operator_expansion} and obtain, for all $1\le r\le n$ and $1\le s\le (n+1)!$,  the expansions
$$
\partial_{r}u_s(\partial_\mbf{x})=\sum_{i=1}^{(n+1)!}\sum_{\sum_k r_k(k+1)\le d_s+1-d_i}
a_{i;\mbf{r}}^{(r,s)}(\mbf{x})u_i(\partial_\mbf{x})D_{1;q}^{r_1}D_{2;q}^{r_2}\cdots D_{n;q}^{r_1}.
$$
Applying both sides of this operator identity to a polynomial $f\in {\mathcal{H}}_{\mbf{x}:q}$ gives
\begin{equation}\label{develop_thm}
\partial_{r}u_s(\partial_\mbf{x})f(\mbf{x})=  \sum_{i=1}^{(n+1)!} 
a_{i;0,0,\,\ldots\,,0}^{(r,s)}(\mbf{x})u_i(\partial_\mbf{x})f(\mbf{x}).
\end{equation}
Now the construction of the monomial basis $u_1(\xi),u_2(\xi),\,\ldots\,,u_{(n+1)!}(\xi)$ used in Theorem~(\ref{critere}) may be carried out in such a manner that the
first $n$ monomials are none other than
$$
\xi_1,\  \xi_2,\, \ldots\, ,\  \xi_n . 
$$
In particular  it follows from  (\ref{develop_thm}) that 
$$
\partial_{r}\partial_{s} f(\mbf{x})=  \sum_{i=1}^{(n+1)!} 
a_{i;0,0,\,\ldots\,,0}^{(r,s)}(\mbf{x})u_i(\partial_\mbf{x})f(\mbf{x}), 
$$
for all $1\le r,s\le n$. We can thus write
$$
\partial_{t}\partial_{r}\partial_{s} f(\mbf{x})=  \sum_{i=1}^{(n+1)!} 
\big(\partial_{t}a_{i;0,0,\,\ldots\,,0}^{(r,s)}(\mbf{x})\big)u_s(\partial_\mbf{x})f(\mbf{x})+
 \sum_{i=1}^{(n+1)!} 
 a_{i;0,0,\,\ldots\,,0}^{(r,s)}(\mbf{x}) \partial_{t}u_s(\partial_\mbf{x})f(\mbf{x}) 
$$
which, using (\ref{develop_thm}) again, gives
\begin{eqnarray*}
\partial_{r}\partial_{s}\partial_{t} f(\mbf{x})
&=&  \sum_{i=1}^{(n+1)!} 
\big(\partial_{t}a_{i;0,0,\,\ldots\,,0}^{(r,s)}(\mbf{x})\big) u_s(\partial_\mbf{x})f(\mbf{x})+
\\
&& \sum_{i=1}^{(n+1)!} 
 a_{i;0,0,\,\ldots\,,0}^{(r,s)}(\mbf{x}) \sum_{j=1}^{(n+1)!}a_{j;0,0,\,\ldots\,,0}^{(t,s)}(\mbf{x})u_j(\partial_\mbf{x})f(\mbf{x}) 
\end{eqnarray*}
We can clearly see  that we have here an algorithm for expressing any derivative of $f(\mbf{x})$
as a linear combinations of the $(n+1)!$ derivatives
\begin{equation}\label{objectif}
u_1(\partial_\mbf{x})f(\mbf{x}) ,\ 
 u_2(\partial_\mbf{x})f(\mbf{x}) ,\  
\ldots, u_{(n+1)!}(\partial_\mbf{x})f(\mbf{x}), 
\end{equation}
where the coefficients of these expansions are products of the coefficients 
$ a_{i;0,0,\,\ldots\,,0}^{(r,s)}(\mbf{x})$ and their successive derivatives. Since differentiation of a rational function
affects its denominator only  by raising its power, it follows that the  singularities in these expansion
are only produced by the zeros of the denominators  of the original coefficients  $ a_{i;0,0,\,\ldots\,,0}^{(r,s)}(\mbf{x})$.
In other words we can safely evaluate these expansions at any point $\mbf{x}'$ which is not in the union of
 the hypersurfaces where the denominators  of  $ a_{i;0,0,\,\ldots\,,0}^{(r,s)}(\mbf{x})$ do vanish.
Since $n$ dimensional space is not the union of  a finite number of hypersurfaces. It follows that
such a point $\mbf{x}'$ can be found with the property that all the derivatives of any $f\in {\mathcal{H}}_{\mbf{x}:q}$
may be expressed in terms of the derivatives in (\ref{objectif}). It follows by applying the   Taylor expansion at $\mbf{x}'$
that $f(\mbf{x})$ itself is determined by these $(n+1)!$  derivatives and thus  
  $(n+1)!$ parameters are sufficient  to uniquely determine an element of $ {\mathcal{H}}_{\mbf{x}:q}$ proving
that its dimension is at most $(n+1)!$ as asserted.
\end{proof}

\section{Last Considerations}
Further computer experiments suggest that a few other interesting facts seem to hold concerning natural variations of the theme discussed in this work. The first one concerns the straight forward extension, of questions raised in section~\ref{sec_regular}, to the so called diagonal situation. Indeed, in the ring $\C[\mbf{x},\mbf{y}]$ of polynomials in two sets of $n$ variables, with $\mbf{y}=y_1,\ldots,y_n$, our experiments suggest that we have

\begin{conj}\label{conj_diag} The  set $\mathcal{D}_n^{\mbf{a}}$ of common polynomial  zeros of the operators
    $$\sum_{i=1}^n a_i\, \partial_{x_i}^k\partial_{y_i}^j,$$
 for all $k,j\in\N$ such that $k+j>0$, is of a bigraded space of dimension $(n+1)^{n-1}$, whenever we have $\mbf{a}=(a_1,\ldots,a_n)$ such that 
\begin{equation}\label{a_i_cond}
    \sum_{k\in K} a_k\not= 0,
 \end{equation}
for all nonempty subsets $K$ of $\{1,\ldots,n\}$.
\end{conj}
Actually, it seems that  this bigraded Hilbert series is independent of the choice of the $a_i$'s, as long condition (\ref{a_i_cond}) holds. Recall here that when all the $a_i$'s are equal to $1$, we get the space $\mathcal{D}_n$ of diagonal harmonics (see \cite{lattice_diagram, haimanhilb}) for which the bigraded Hilbert series is explicitly known. Thus we could refine the conjecture as stating that $\mathcal{D}_n^{\mbf{a}}$ and $\mathcal{D}_n$ are isomorphic as graded vector spaces.  Among the supporting facts for Conjecture~\ref{conj_diag}, one can observe that an immediate consequence of Theorem~\ref{critere}  is that the graded space $\mathcal{H}_n^{\mbf{a}}$ of polynomials $f(\mbf{x})$ such that  
    $$\sum_{i=1}^n a_i\, \partial_{x_i}^k f(\mbf{x})=0,$$
 for all $k\geq 1$, is of dimension $n!$ with Hilbert series equal to 
  $$(1+t)(1+t+t^2)\cdots (1+t+\ldots + t^{n-1}).$$
It particular its homogeneous component of maximal degree ($=\binom{n}{2}$) is one dimensional. Let us choose some vector $\Delta_n^{\mbf{a}}(\mbf{x})$ whose span is this maximal degree component. Observing that it is closed under derivatives and that its dimension is $n!$, we must conclude that $\mathcal{H}_n^{\mbf{a}}$ coincides with the linear span of all higher order partial derivatives of $\Delta_n^{\mbf{a}}(\mbf{x})$. We formulate this as $\mathcal{H}_n^{\mbf{a}}=\mathcal{L}_\partial (\Delta_n^{\mbf{a}}(\mbf{x}))$. Now it develops that there is a description for the space $\mathcal{D}_n$ of a similar flavour, as is shown by Haiman in~\cite{haimanhilb}. To state the relevant result, let us introduce the operators
   $$E_k:=\sum_{i=1}^n y_i\,\partial x_i^k,$$
and use the notation $\mathcal{L}_{\partial,E} (f(\mbf{x},\mbf{y}))$ for the smallest vector space containing a given polynomial $f(\mbf{x},\mbf{y})$ and which is closed under partial derivatives and applications of the operators $E_k$. Among the various interestring results of Haiman is the fact that $\mathcal{D}_n=\mathcal{L}_{\partial,E} (\Delta_{\S_n}(\mbf{x}))$. This relies in part on the property of $\mathcal{D}_n$ of being closed under applications of the $E_k$'s. The proof of this property is easily adapted to the context of the spaces $\mathcal{D}_n^{\mbf{a}}$, under no special conditions on $\mbf{a}$, and it is quite clear that $\mathcal{D}_n^{\mbf{a}}$ is always closed under partial derivatives. To sum up, we arrive at the conclusion that conditions (\ref{conditions_ai}) imply that $\mathcal{L}_{\partial,E} (\Delta_n^{\mbf{a}}(\mbf{x}))$ is a subspace of $\mathcal{D}_n^{\mbf{a}}$. In fact, under the hypothesis that Conjecture~\ref{conj_diag} holds, we should have equality. An explicit calculation of the polynomial $\Delta_n^{\mbf{a}}(\mbf{x})$, say for $n=3$ and the $a_i$'s generic, makes apparent some of its nice properties:
\begin{eqnarray*}
\Delta_3^{\mbf{a}}(\mbf{x})&=&{\frac {2\,{x_1}^{2}x_2}{ a_1\,(a_1+a_2 ) }}
-{\frac {2\,{x_1}^{2}x_3}{ a_1\,(a_1+a_3 ) }}
-{\frac {2\,{x_2}^{2}x_1}{ a_2\,(a_1+a_2 ) }}
+{\frac {2\,{x_2}^{2}x_3}{ a_2\, ( a_2+a_3 ) }}\\
&&\qquad 
 +{\frac {2\,{x_3}^{2}x_1}{ a_3\,( a_1+a_3 )}}
 -{\frac {2\,{x_3}^{2}x_2}{ a_3\,( a_2+a_3 ) }}
 +{\frac {2\, (a_3-a_2)\, {x_1}^{3}}{3\,a_1\,(a_1+a_2)(a_1+a_3)}}\\
&&\qquad 
+{\frac {2\, (a_1-a_3)\, {x_2}^{3}}{3\,a_2\,(a_1+a_2)(a_2+a_3)}}
+{\frac { 2\,(a_2-a_1)\, {x_3}^{3}}{3\,a_3\,(a_1+a_3)(a_2+a_3)}}
 \end{eqnarray*}
 which is here normalized so that it specializes to the classical Vandermonde determinant when all the $a_i$ are set $1$. 
An intriguing  observation about this polynomial can be made if one considers the $a_i$'s as variables on which $\S_n$  acts just as it does on the $x_i$'s. Then one has the expression
   $$\Delta_3^{\mbf{a}}(\mbf{x})= 
   R^{\pm}_3\left({\frac {2\,{x_1}^{2}x_2}{ a_1\,(a_1+a_2 ) }}
+{\frac {( a_3-a_2)\, {x_1}^{3}}{3\,a_1\,( a_1+a_2) ( a_1+a_3) }}\right)
$$
where $R^\pm_n:=\sum_{\sigma\in\S_n} \mathrm{sign}(\sigma)\,\sigma$ stands for the anti-symmetrization operator. 

Another interesting experimental observation concerning the space of common zeros of $D_1$ and $D_2$ with general operators
\begin{eqnarray*}
  D_1&:=&\sum_{i=1}^n a_{i}\,x_i \partial_{i}^{2} + b_{i}\, \partial_{i},\\
  D_2&:=&\sum_{i=1}^n c_{i}\,x_i \partial_{i}^{3} + d_{i}\, \partial_{i}^2,
\end{eqnarray*}
is that there seem to be conditions, similar to (\ref{a_i_cond}), for which this space is always $n!$-dimensionnal.

\end{document}